\documentclass[12pt]{article}
\usepackage{amsmath,amsfonts,amssymb,amsthm}
\usepackage{mathrsfs,hyperref}
\usepackage{color}
\usepackage[square,numbers]{natbib}

\newtheorem{proposition}{Proposition}[section]
\newtheorem{theorem}[proposition]{Theorem}
\newtheorem{corollary}[proposition]{Corollary}
\newtheorem{lemma}[proposition]{Lemma}
\newtheorem{remark}[proposition]{Remark}
\newtheorem{definition}[proposition]{Definition}

\numberwithin{equation}{section}

\newcommand{\nc}{\newcommand}
\newcommand{\rnc}{\renewcommand}

\nc{\I}{{\bf 1}}

\nc{\bA}{\mathbf A}
\nc{\cA}{\mathcal A}
\nc{\cB}{\mathcal B}
\nc{\cC}{\mathcal C}
\nc{\cD}{\mathcal D}
\nc{\disj}{\mathrm{disj}}
\nc{\E}{\mathbb E}
\nc{\fF}{\mathfrak F}
\nc{\cL}{\mathcal L}
\nc{\dom}{\mathrm{dom}}
\nc{\lip}{\mathrm{Lip}}
\nc{\cM}{\mathcal M}
\nc{\N}{\mathbb N}
\nc{\bN}{\mathbf N}
\nc{\cN}{\mathcal N}
\nc{\sN}{\mathsf N}
\rnc{\P}{\mathbb P}
\nc{\bP}{\mathbf P}
\nc{\cP}{\mathcal P}
\nc{\R}{\mathbb R}
\nc{\cS}{\mathcal S}
\nc{\cU}{\mathcal U}
\nc{\bx}{\mathbf x}
\nc\X{\mathbb X}
\nc{\bX}{\mathbf X}
\nc{\cX}{\mathcal X}
\nc{\by}{\mathbf y}
\nc{\Y}{\mathbb Y}
\nc{\bY}{\mathbf Y}
\nc{\cY}{\mathcal Y}
\nc{\Z}{\mathbb Z}
\nc{\bZ}{\mathbf Z}

\nc{\Var}{\mathrm{Var}}
\nc{\pp}[4]{p _{#1,#2}^{#3,#4}}
\nc{\maxp}{\bar{p}}
\nc{\maxsump}{\tilde{p}}
\nc{\size}{X}

\makeatletter
\def\keywords{\xdef\@thefnmark{}\@footnotetext}
\makeatother

\title{
	Multivariate Poisson and Poisson process approximations with applications
	\\
	to Bernoulli sums and $U$-statistics
	}
\author{
	
	Federico Pianoforte\footnotemark[1]
	\, and \,
	Riccardo Turin\footnotemark[2]
	}

\begin{document}
	\footnotetext[1]{federico.pianoforte@stat.unibe.ch;  University of Bern, Institute of Mathematical Statistics and Actuarial Science.}
	\footnotetext[2]{riccardo.turin@stat.unibe.ch;  University of Bern, Institute of Mathematical Statistics and Actuarial Science.}
	\maketitle

\begin{abstract}
	This article derives quantitative limit theorems for multivariate Poisson and Poisson process approximations. 
	Employing the solution of Stein's equation for Poisson random variables, we obtain an explicit bound for the multivariate Poisson approximation of random vectors in the Wasserstein distance.
	The bound is then utilized in the context of point processes to provide a Poisson process approximation result in terms of a new metric called $d_\pi$, stronger than the total variation distance, defined as the supremum over all Wasserstein distances between random vectors obtained by evaluating the point processes on arbitrary collections of disjoint sets.
	As applications, the multivariate Poisson approximation of the sum of $m$-dependent Bernoulli random vectors, the Poisson process approximation of point processes of $U$-statistic structure and the Poisson process approximation of point processes with Papangelou intensity are considered.
	Our bounds in $d_\pi$ are as good as those already available in the literature.
\end{abstract}
	

\section{Introduction and main results} 

Let $\bX=(X_1,\dots ,X_d)$ be an integrable random vector taking values in $\N_0^d$, $d\in\N$, where $\N_0=\N\cup \{0\}$,	and let $\bP=(P_1,\dots ,P_d)$ be a Poisson random vector,
that is, a random vector  with independent and Poisson distributed components.
The first contribution of this paper is  an upper bound on the Wasserstein distance
\begin{equation*}
\label{eq:WD}
	d_{W}(\bX,\bP)=\sup_{g\in \lip^d(1)}\big|\E[g(\bX)]-\E[g(\bP)]\big|
\end{equation*}
between   $\bX$  and   $\bP$, where $\lip^d(1)$ denotes the set of Lipschitz functions $g:\N_0^d \rightarrow \R$ with Lipschitz constant bounded by $1$ with respect to the metric induced by the $1$-norm $|\bx|_1= \sum_{i=1}^d |x_i|$, for $\bx=(x_1,\dots,x_d)\in\R^d$.

The accuracy of the multivariate Poisson approximation has mostly been studied in terms of the total variation distance; among others we mention \cite{MR972770,bar1988,MR2205332,MR1163825,ekanaviius2020compound,MR1978097,MR3647296}.
In contrast, we consider the Wasserstein distance.
Note that, since the indicator functions defined on $\N_0^d$ are Lipschitz continuous, for random vectors in $\N_0^d$ the Wasserstein distance dominates the total variation distance, and
it is not hard to find sequences that converge in total variation distance but not in Wasserstein distance.
Our goal is to extend the approach developed in \cite{pianoforte2021poisson} for the Poisson approximation of random variables to the multivariate case.

Throughout the paper, for any $\bx=(x_1,\dots,x_d)\in\R^d$ and index $1\leq j\leq d$, we denote by $x_{1:j}$ and $x_{j:d}$ the subvectors $(x_1,\dots,x_j)$ and $(x_j,\dots,x_d)$, respectively.
\begin{theorem}
\label{theo:1}
	Let $\bX=(X_1,\dots ,X_d)$ be an integrable random vector with values in $\N_0^d$, \break $d\in\N$, and let $\bP=(P_1,\dots , P_d)$	be a Poisson random	vector with $\E[\bP]=(\lambda_1,\dots,\lambda_d)\in [0,\infty)^d$.
	For $1\leq i\leq d$, consider any random vector $\bZ^{(i)}=(Z^{(i)}_1,\dots,Z^{(i)}_i)$ in $\Z^i$ defined on the same probability space as $\bX$, and define
	\begin{equation}
	\label{eq: distr of Z}
		q^{(i)}_{m_{1:i}}=m_i\P\big(X_{1:i}=m_{1:i}\big)-\lambda_i\P\big(X_{1:i}+\bZ^{(i)}=(m_{1:i-1}, m_i -1)\big)
	\end{equation}
	for $m_{1:i}\in\N_0^{i}$ with $m_i\neq0$.
	Then,
	\begin{equation}
	\label{eq: thm1}
		d_W(\bX,\bP)\leq\sum_{i=1}^d\left(\lambda_i \E\big|Z^{(i)}_i\big|
		+2\lambda_i\sum_{j=1}^{i-1}\E\big|Z^{(i)}_j\big|
		+\sum_{\substack{m_{1:i}\in\N_0^{i}\\m_i\neq0}}	\left|q^{(i)}_{m_{1:i}}\right|\right).
	\end{equation}
\end{theorem}
For a random variable $X$, Equation \eqref{eq: distr of Z} corresponds to the condition required in \cite[Theorem 1.2]{pianoforte2021poisson}.
There, sharper bounds on the Wasserstein distance for the case of random variables are shown.
However, Theorem~\ref{theo:1} tackles the case of random vectors instead of just considering random variables.

In order to give an interpretation of the hypothesis in Theorem~\ref{theo:1}, let us consider the random vectors
\begin{equation}
\label{eq:size0}
	\bY^{(i)}=\left(X_{1:i-1},X_i+1\right)+\bZ^{(i)}
	,\quad i=1,\dots,d,
\end{equation}
with $\bX$ and $\bZ^{(i)}$ defined as in Theorem \ref{theo:1}.
Under the condition $\P(X_{1:i}+\bZ^{(i)}\in\N_0^i)=1$,
a sequence of real numbers $q^{(i)}_{m_{1:i}}$, $m_{1:i}\in\N_0^{i}$ with $m_i\neq0$,
satisfies Equation \eqref{eq: distr of Z} if and only if
\begin{equation}
\label{eq:size}
	\E[X_i f(X_{1:i})]=\lambda_i\E[f(\bY^{(i)})]
	+\sum_{m_{1:i}\in\N_0^{i},\,m_i\neq0} q^{(i)}_{m_{1:i}}f(m_{1:i})
\end{equation}
for all functions $f:\N_0^i\to\R$ such that $\E\left|X_i f(X_{1:i})\right|<\infty$.
When the $q^{(i)}_{m_{1:i}}$ are all zero and $\E[X_i]=\lambda_i$,
the condition $\P(X_{1:i}+\bZ^{(i)}\in\N_0^i)=1$ is satisfied, as can be seen by taking the sum over $m_{1:i}\in\N_0^{i}$ with $m_i\neq0$ in \eqref{eq: distr of Z}.
In this case, \eqref{eq:size} becomes
\begin{equation}
\label{eq:size2}
	\E[X_i f(X_{1:i})]=\E[X_i]\E[f(\bY^{(i)})].
\end{equation}
Recall that, for a random variable $\size\geq 0$ with mean $\E[\size]>0$, a random variable $\size^s$ defined on the same probability space as $\size$ is a size bias coupling of $\size$ if it satisfies
\begin{equation}
\label{eq:sizeBias-1d}
	\E[\size f(\size)]=\E[\size]\E[f(\size^s)]
\end{equation}
for all measurable $f:\R\to\R$ such that $\E\left|\size f(\size)\right|<\infty$.
Therefore,  if the $q^{(i)}_{m_{1:i}}$ are all zero and $\E[X_i]=\lambda_i$ for $i=1,\dots,d$, the random vectors $\bY^{(1)},\dots,\bY^{(d)}$ can be seen as the size bias coupling of $\bX$, as they are defined on the same probability space as $\bX$ and satisfy \eqref{eq:size2}, which corresponds to \eqref{eq:sizeBias-1d} in the one dimensional case, that is when $i=d=1$.
Note that this suggests a definition of size bias coupling of random vectors that is slightly different from the one introduced in \cite{MR1371949}. 
Following this interpretation, when $\E[\bX]=(\lambda_1,\dots,\lambda_d)$ and the random vectors $\bZ^{(i)}$ are chosen such that the $q_{m_{1:i}}^{(i)}$ are not zero, we can think of the random vectors $\bY^{(i)}$ defined in \eqref{eq:size0}  as an approximate size bias coupling of $\bX$, where instead of assuming that $\bY^{(i)}$ satisfies \eqref{eq:size2} exactly, we allow error terms $q^{(i)}_{m_{1:i}}$.
This is an important advantage of Theorem~\ref{theo:1}, since one does not need to find an exact size bias coupling (in the sense of \eqref{eq:size2}), it only matters that both error terms $q_{m_{1:i}}^{(i)}$ and random vectors $\bZ^{(i)}$ are sufficiently small. 

The second main contribution of our work concerns Poisson process approximation of point processes with finite intensity measure.
For a  point processes $\xi$ and a Poisson process $\eta$ on a measurable space $\X$ with finite intensity measure,
Theorem~\ref{theo:1} provides bounds on the Wasserstein distance
\begin{equation*}
	d_W((\xi(A_1),\dots,\xi(A_d)), (\eta(A_1),\dots,\eta(A_d))\,,
\end{equation*}
where $A_1,\dots,A_d$ are measurable subsets of $\X$.
This allows for a way to compare the distributions of $\xi$ and $\eta$, by taking the supremum of the Wasserstein distances between the point processes evaluated on arbitrary collections $(A_1,\dots,A_d)$ of disjoint sets.
More precisely, let $(\X,\cX)$ be a measurable space and define $\sN_\X$ as the collection of all $\sigma$-finite counting measures.
The set $\sN_\X$ is equipped with the $\sigma$-field $\cN_\X$ generated by the collection of all subsets of $\sN_\X$ of the form
\begin{equation*}
	\{\nu\in\sN_\X\,:\,\nu(B)=k\},\quad B\in\cX,\,k\in\N_0.
\end{equation*}
This means that $\cN_\X$ is the smallest $\sigma$-field on $\sN_\X$ that makes the map $\nu\mapsto \nu(B)$ measurable for all $B\in\cX$.
A point process $\xi$ on $\X$ is a random element in $(\sN_\X,\cN_\X)$.
The intensity of $\xi$ is the measure $\lambda$ on $(\X,\cX)$ defined by $\lambda(B)=\E[\xi(B)], B\in\cX$.
When a point process $\xi$ has finite intensity measure $\lambda$, for any choice of subsets $A_1,\dots,A_d\in\cX$, the random vector $(\xi(A_1),\dots,\xi(A_d))$ takes values in $\N_0^d$ (almost surely).
Thus, we define a metric in the space of point processes with finite intensity measure in the following way.
\begin{definition}
\label{def:OUR}
	Let $\xi$ and $\zeta$ be point processes on $\X$ with finite intensity measure.
	The distance $d_\pi$ between $\xi$ and $\zeta$ is defined as
	\begin{equation*}
		d_{\pi}(\xi,\zeta)=\sup_{(A_1,\dots,A_d)\in \cX^d_{\disj},\,d\in \N }
		d_{W}\big((\xi(A_1),\dots,\xi(A_d)),(\zeta(A_1),\dots,\zeta(A_d))\big),
	\end{equation*}
	where
	\begin{equation*}
		\cX^d_{\disj}=\{(A_1,\dots , A_d)\in \cX^d\,:\, A_i\cap A_j =\emptyset, i\neq j\}.	
	\end{equation*}
\end{definition}
That $d_\pi$ is a probability distance between the distributions of point processes follows immediately from its definition and, e.g., \cite[Proposition 2.10]{MR3791470}.

To the best of our knowledge, this is the first time the distance $d_\pi$ is defined and employed in Poisson process approximation.
We believe that it is possible to extend $d_\pi$ to larger classes of point processes by restricting $\cX^d_{\disj}$ to suitable families of sets.
However, this falls out of the scope of this paper, and it will be treated elsewhere.
Let us now state our main theoretical result on Poisson process approximation.
\begin{theorem}
\label{theo:2}
	Let $\xi$ be a point process on $\X$ with finite intensity measure, and let $\eta$ be a Poisson process on $\X$  with finite intensity measure $\lambda$.
	For any $i$-tuple $(A_1,\dots,A_i)\in \cX^{i}_\disj$ with $i\in\N$, consider a random vector  $\bZ^{A_{1:i}}=(Z^{A_{1:i}}_1,\dots,Z^{A_{1:i}}_i)$ defined on the same probability space as $\xi$ with values in $\Z^i$, and define
	\begin{equation}
	\label{eq:hpThPP}
		\begin{split}
		&q^{A_{1:i}}_{m_{1:i}}=m_i \P\big((\xi(A_1),\dots ,\xi (A_i)) = m_{1:i}\big)
		\\
		&\hspace{4cm}-\lambda(A_i)\mathbb{P}\big((\xi(A_1),\dots ,\xi (A_i)) + \bZ^{A_{1:i}} =(m_{1:i-1}, m_i-1)\big)
		\end{split}
	\end{equation}
	for $m_{1:i}\in\N_0^i$ with	$m_i\neq0$.
	Then, 
	\begin{equation*}
		d_{\pi}(\xi, \eta)\leq\sup_{(A_1,\dots , A_d)\in\cX^d_{\disj},d\in\N}\,\,
		\sum_{i=1}^d \left( \sum_{\substack{m_{1:i}\in\N_0^{i}\\m_i\neq0}}\left|q^{A_{1:i}}_{m_{1:i}}\right|
		+ 2\lambda(A_i) \sum_{j=1}^i \E\big| Z^{A_{1:i}}_{j} \big|  \right).
	\end{equation*}
\end{theorem}
The Poisson process approximation has mostly been studied in terms of the total variation distance in the literature; see e.g.\ \cite{MR1092983,bar1988,MR1190904,Brow:Xia2001,ChenXia2004,Schu2009,Schu:Stu2014} and references therein.
In contrast, \cite{decr:schu:taele, decr:vass2018} deal with Poisson process approximation using the Kantorovich–Rubinstein distance.
In Proposition~\ref{prop:TVD<OUR}, we establish that the total variation distance 
\begin{equation*}
	d_{TV}(\xi,\zeta)=\sup_{B\in\cN_\X}	|\P(\xi \in B)-\P(\zeta\in B)|
\end{equation*}
between two point processes $\xi$ and $\zeta$ on $\X$ with finite intensity measure is bounded from above by $d_{\pi}(\xi, \zeta)$.
Moreover, since $d_\pi(\xi,\zeta)\geq | \mathbb{E}[\xi(\mathbb X)] - \mathbb{E}[\zeta (\mathbb X)]|$, Example 2.2 in \cite{decr:schu:taele} provides a sequence of point processes $(\zeta_n)_{n\geq 1}$ that converges in total variation distance to a point process $\zeta$ even though $d_\pi(\zeta_n,\zeta)\to\infty$ as $n$ goes to infinity.
This shows that $d_\pi$ is stronger than $d_{TV}$ in the sense that convergence in $d_\pi$ implies convergence in total variation distance, but not vice versa.
The Kantorovich-Rubinstein distance between two point processes $\xi$ and $\zeta$ is defined as the optimal transportation cost between their distributions, when the cost function is the total variation distance between measures; see \cite[Equation 2.5]{decr:schu:taele}.
When the configuration space $\X$ is a locally compact second countable Hausdorff space (lcscH), which is indeed the case considered in \cite{decr:schu:taele} and \cite{decr:vass2018}, the Kantorovich duality theorem (\cite[Theorem 5.10]{Vill2009}) yields an equivalent definition for this metric:
\begin{equation*}
	d_{KR}(\xi,\zeta)=\sup\left| \E[h(\xi)]-\E[h(\zeta)]\right|
\end{equation*} 
where the supremum runs over all measurable functions $h:\sN_\X\to\R$ that are $1$-Lipschitz with respect to the total variation distance between measures and make $h(\xi)$ and $h(\zeta)$ integrable.
For a lcscH space $\X$, we prove in Lemma~\ref{lemma-KR,pi-rel} that $d_\pi\leq 2 d_{KR}$. 	The constant $2$ in this  inequality  cannot be improved, as shown by the following simple example:
let $\X=\{a,b\}$ with $\cX=\{\emptyset,\{a\},\{b\},\X\}$ and let $\delta_a$ and $\delta_b$ be deterministic point processes corresponding to the Dirac measures centered at $a$ and $b$, respectively.
	Since the function $g:(x_1,x_2)\mapsto x_1-x_2$ is $1$-Lipschitz, it follows
	\begin{equation*}
		d_\pi(\delta_a,\delta_b)\geq
		|\,g(\delta_a(\{a\}),\delta_a(\{b\}))- g(\delta_b(\{a\}),\delta_b(\{b\}))\,|=2.
	\end{equation*}
On the other hand, $d_{KR}$ is bounded by the expected total variation distance between the two counting measures, thus $d_{KR}(\delta_a,\delta_b)\leq 1$.
Hence, in this case $d_\pi(\delta_a,\delta_b)=2d_{KR}(\delta_a,\delta_b)$. 	

It remains an open problem whether the distances $d_\pi$ and $d_{KR}$  are equivalent or not.
It is worth mentioning that our general result, Theorem \ref{theo:2}, permits to study the Poisson process approximation of point processes on any measurable space.
Hence,  Theorem \ref{theo:2} can be used to obtain approximation results for point processes also when the notion of weak convergence is not defined.

To demonstrate the versatility of our general main results, we apply them to several examples. 
In Subsection~\ref{Sec:application-eq:mult-dstr}, we approximate the sum of Bernoulli random vectors by a Poisson random vector.
This problem has mainly been studied in terms of the total variation distance and under the assumption that the Bernoulli random vectors are independent (see\,\,e.g.\,\cite{MR1701409}).
We derive an explicit approximation result in the Wasserstein distance for the more general case of $m$-dependent Bernoulli random vectors.

In Subsections~\ref{sec:pap-int}  and \ref{Sec:U-stat}, we apply Theorem~\ref{theo:2} to obtain explicit Poisson process approximation results for point processes with Papangelou intensity and point processes of Poisson $U$-statistic structure.
The latter are point processes that, once evaluated on a measurable set, become Poisson $U$-statistics.
Analogous results were already proven for the Kantorovich-Rubinstein distance in \cite[Theorem 3.7]{decr:vass2018} and  \cite[Theorem 3.1]{decr:schu:taele}, under the additional condition that the configuration space $\X$ is lcscH.
It is interesting to note that the proof of our result for point processes with Papangelou intensity employs Theorem~\ref{theo:2}  with $\bZ^{A_{1:i}}$ set to zero for all $i$, while for point processes of $U$-statistic structure, we find $\bZ^{A_{1:i}}$ such that Equation \eqref{eq:hpThPP} in Theorem~\ref{theo:2} is satisfied with $q^{A_{1:i}}_{m_{1:i}}\equiv 0$ for all collections of disjoint sets.

The proof of Theorem~\ref{theo:1} is based on the Chen-Stein method applied to each component of the random vectors and the coupling in \eqref{eq: distr of Z}.
In the proof of Theorem~\ref{theo:2} we mimic the approach used 
in \cite{MR972770} to prove Theorem 2,
as we derive the process bound as a consequence of the $d$-dimensional bound.

Before we discuss the applications in Section 3, we prove our main results in the next
section.

\section{Proofs of the main results}
\label{Sec:proofs}

Throughout this section, $\bX=(X_1,\dots, X_d)$ is an integrable random vector with values in $\N_0^d$ and	$\bP=(P_{1},\dots , P_d)$ is a Poisson random vector with mean $\E[\bP]=(\lambda_1,\dots,\lambda_d)\in [0,\infty)^d$.
Without loss of generality we assume that $\bX$ and $\bP$ are independent and defined on the same probability space $(\Omega,\fF,\P)$.
We denote by $\lip^d(1)$ the collection of Lipschitz functions $g:\N_0^d\to\R$ with respect to the metric induced by the $1$-norm and Lipschitz constant bounded by $1$, that is
\begin{equation*}
	|g(\bx)-g(\by)|\leq|\bx-\by|_1=\sum_{i=1}^d |x_{i}-y_{i}|,\quad \bx,\by\in\N_0^d.
\end{equation*}	
Clearly, this family of functions contains the $1$-Lipschitz functions with respect to the Euclidean norm.
For $d=1$, we use the convention  $\lip(1)=\lip^1(1)$.

For any fixed $g\in \mathrm{Lip}(1)$, a solution of Stein's equation for Poisson distribution is a real valued function $\widehat{g}^{(\lambda)}:\N_0\to\R$ that satisfies 
\begin{equation}
\label{eq:stein-poisson}
	\lambda\widehat{g}^{(\lambda)}(i+1)-i\widehat{g}^{(\lambda)}(i)
		=g(i)-\E [g(P_\lambda)],\quad i\in\N_0 ,
\end{equation}
where $P_\lambda$ is a Poisson random variable	with mean $\lambda\geq 0$.
We fix for convenience the initial condition $\widehat{g}^{(\lambda)}(0)=0$.
With this assumption, the function $\widehat{g}^{(\lambda)}$ is unique and may be obtained by solving \eqref{eq:stein-poisson} recursively  on $i$.
An explicit expression for  this solution is given in  \cite[Lemma 1]{bar1988}.
The following lemma is a direct consequence of \cite[Theorem 1.1]{MR2274850}
(Note that the case $\lambda=0$ is trivial).
\begin{lemma}
\label{lem:stein-poisson}
	For any $\lambda\geq0$ and $g\in\lip(1)$, let $\widehat{g}^{(\lambda)}$ be the solution of the Stein equation \eqref{eq:stein-poisson} with  initial condition  $\widehat{g}^{(\lambda)}(0)=0$.
	Then,
	\begin{equation}
	\label{eq:magic-factors}
		\sup_{i\in\N_0}\left|\widehat{g}^{(\lambda)}(i)\right|\leq 1
		\quad \text{and}\quad
		\sup_{i\in\N_0}\left|\widehat{g}^{(\lambda)}(i+1)-\widehat{g}^{(\lambda)}(i)\right|\le 1.
	\end{equation}  
\end{lemma}
Recall that, for any $\bx=(x_1,\dots,x_d)\in\R^d$ and some index $1\leq j\leq d$, we write $x_{1:j}$ and $x_{j:d}$ for the subvectors $(x_1,\dots,x_j)$ and $(x_j,\dots,x_d)$, respectively.
For $g\in\lip^d(1)$, let $\widehat{g}^{(\lambda)}_{x_{1:i-1}|x_{i+1:d}}$ denote	the solution to \eqref{eq:stein-poisson} for the Lipschitz function	$g(x_{1:i-1},\cdot\ ,x_{i+1:d})$ with fixed $x_{1:i-1}\in\N_0^{i-1}$ and $x_{i+1:d}\in\N_0^{d-i}$. 
Since $\widehat{g}^{(\lambda)}$ takes vectors from $\N_0^d$ as input, we do not need to worry about measurability issues.
The following proposition is the first building block for the proof of Theorem~\ref{theo:1}.
\begin{proposition}
\label{prop:1}
	For any $g\in\lip^d(1)$, 
	\begin{equation*}
		\E[g(\bP) - g(\bX)]=\sum_{i=1}^d\E\left[X_i\widehat{g}^{(\lambda_i)}_{X_{1:i-1}|P_{i+1:d}}(X_i)
		-\lambda_i\widehat{g}^{(\lambda_i)}_{X_{1:i-1}|P_{i+1:d}}(X_i+1)\right].
	\end{equation*}
\end{proposition}
\begin{proof}[Proof of Proposition~\ref{prop:1}]
	First, observe that
	\begin{equation}
	\label{eq: stepProofProp1}
		\E\left[g(\bP)-g(\bX)\right]
		=\sum_{i=1}^d\E\left[g(X_{1:i-1},P_{i:d})-g(X_{1:i},P_{i+1:d})\right].
	\end{equation}
	The independence of $P_i$ from $P_{i+1:d}$ and $X_{1:i}$ implies
	\begin{equation*}
		\E\big[g(X_{1:i-1},P_{i:d})- g(X_{1:i},P_{i+1:d})	\big]
		=\E\big[\E^{P_i}[g(X_{1:i-1},P_{i:d})]	- g(X_{1:i},P_{i+1:d})\big],
	\end{equation*}
	where $\E^{P_i}$ denotes the expectation	with respect to the random variable $P_i$.
	From the definition of $\widehat{g}^{(\lambda_i)}_{x_{1:i-1}|x_{i+1:d}}$ with $x_{1:i-1}=X_{i:i-1}$ and $x_{i+1:d}=P_{i+1:d}$,	it follows 
	\begin{equation*}
		\E^{P_i}[g(X_{1:i-1}, P_{i:d})]- g(X_{1:i},P_{i+1:d})
		=X_i\widehat{g}^{(\lambda_i)}_{X_{1:i-1}|P_{i+1:d}}(X_i)
		-\lambda_i\widehat{g}^{(\lambda_i)}_{X_{1:i-1}|P_{i+1:d}}(X_i+1)
	\end{equation*}
	for all $i=1,\dots, d$.
	Together with \eqref{eq: stepProofProp1}, this leads to the desired conclusion.
\end{proof}
\begin{proof}[Proof of Theorem~\ref{theo:1}]
	In view of Proposition~\ref{prop:1}, it suffices to bound
	\begin{equation*}
		\left|\E\left[X_i \widehat{g}^{(\lambda_i)}_{X_{1:i-1}|P_{i+1:d}}(X_i)
		-\lambda_i\widehat{g}^{(\lambda_i)}_{X_{1:i-1}|P_{i+1:d}}(X_i+1)\right]\right|
		,\quad i=1,\dots,d\,.
	\end{equation*}
	For the remaining of the proof,	the index $i$ is fixed and we omit the superscript $(i)$ in $Z_{i:d}^{(i)}$ and $q^{(i)}_{m_{1:i}}$.
	Define the function
	\begin{equation*}
		h(X_{1:i})=\E\left[\widehat{g}^{(\lambda_i)}_{X_{1:i-1}|P_{i+1:d}}(X_i)\,\big|\,X_{1:i}\right],
	\end{equation*}
	where $\E[\ \cdot\ |\, Y]$ denotes the conditional expectation with respect to a random element $Y$.
	With the convention	$\widehat{g}^{(\lambda_i)}_{m_{1:i-1}|m_{i+1:d}}(m_i)=0$  if $m_{1:d}\notin\N_0^d$ or $m_i=0$, it follows from \eqref{eq: distr of Z} that
	\begin{align*}
		&\E\left[X_i \widehat{g}^{(\lambda_i)}	_{X_{1:i-1}|P_{i+1:d}}(X_i)\right]
		=\E[X_i h(X_{1:i})]
		=\sum_{m_{1:i}\in\N_0^{i}}m_i h(m_{1:i})\P(X_{1:i}=m_{1:i})
		\\
		&=\sum_{\substack{m_{1:i}\in\N_0^{i}\\m_i\neq0}}h(m_{1:i})q_{m_{1:i}}
		+\lambda_i\sum_{\substack{m_{1:i}\in\N_0^{i}\\m_i\neq0}}h(m_{1:i})
		\P\left(X_{1:i}+Z_{1:i}=(m_{1:i-1},m_{i}-1)\right)
		\\
		&=\sum_{\substack{m_{1:i}\in\N_0^{i}\\m_i\neq0}}h(m_{1:i})q_{m_{1:i}}
		+\lambda_i\E\left[	\widehat{g}^{(\lambda_i)}_{X_{1:i-1}+Z_{1:i-1}|P_{i+1:d}}(X_i+Z_i+1)\right].
	\end{align*}
	Since $|h(X_{1:i})|\leq 1$ by \eqref{eq:magic-factors}, the triangle inequality establishes
	\begin{equation}
	\label{eq:bound1}
		\left|\E\left[X_i \widehat{g}^{(\lambda_i)}_{X_{1:i-1}|P_{i+1:d}}(X_i)
		-\lambda_i\widehat{g}^{(\lambda_i)}_{X_{1:i-1}|P_{i+1:d}}(X_i+1)\right]\right|
		\leq\sum_{\substack{m_{1:i}\in\N_0^{i}\\m_i\neq0}}\left|q_{m_{1:i}}\right|+\lambda_i(H_1+H_2),
	\end{equation}
	with
	\begin{equation*}
		H_1=\left|\E\left[\widehat{g}^{(\lambda_i)}	_{X_{1:i-1}+ Z_{1:i-1}|P_{i+1:d}}(X_i + Z_i + 1)
		-\widehat{g}^{(\lambda_i)}_{X_{1:i-1}+ Z_{1:i-1}|P_{i+1:d}}(X_i+1)\right]\right|
	\end{equation*}
	and
	\begin{equation*}
		H_2=\left|\E\left[\widehat{g}^{(\lambda_i)}_{X_{1:i-1}+ Z_{1:i-1}|P_{i+1:d}}(X_i+1)
		-\widehat{g}^{(\lambda_i)}_{X_{1:i-1}|P_{i+1:d}}(X_i+1)\right]\right|.
	\end{equation*}
	The inequalities in \eqref{eq:magic-factors} guarantee
	\begin{equation*}
		H_1\leq \E|Z_i|	\quad\text{and}\quad
		H_2\leq2\P(Z_{1:i-1}\neq0)\leq\sum_{j=1}^{i-1} 2\P(Z_j\neq0)\leq 2\sum_{j=1}^{i-1} \E|Z_j|.
	\end{equation*}
	Combining \eqref{eq:bound1} with the	bounds for $H_1$ and $H_2$, and summing over $i=1,\dots ,d$ concludes the proof.
\end{proof}
\begin{remark}
It follows directly from the previous proof that the bound \eqref{eq: thm1} in Theorem~\ref{theo:1} can be improved in the following way:
	\begin{equation*}
		d_W(\bX,\bP)\leq\sum_{i=1}^d\left(	\lambda_i \E\big|Z^{(i)}_i\big|
		+2\lambda_i \P\big(Z^{(i)}_{1:i-1}\neq 0\big)
		+\sum_{\substack{m_{1:i}\in\N_0^{i}\\m_i\neq0}}	\left|q^{(i)}_{m_{1:i}}\right|\right).
	\end{equation*}
\end{remark}
Next, we derive Theorem~\ref{theo:2} from Theorem~\ref{theo:1}.
\begin{proof}[Proof of Theorem~\ref{theo:2}]
	Let $d\in\N$ and $\bA=(A_1,\dots,A_d)\in\cX_\disj^d$.
	Define
	\begin{equation*}
		\bX^\bA=(\xi(A_1),\dots,\xi(A_d))
		\quad\text{and}\quad
		\bP^\bA=(\eta(A_1),\dots,\eta(A_d)),
	\end{equation*}
	where $\bP^\bA$ is a Poisson random vector with mean $\E[\bP^\bA]=(\lambda(A_1),\dots,\lambda(A_d))$.
	By Theorem~\ref{theo:1} with $\bZ^{(i)}=\bZ^{A_{1:i}}$, we obtain
	\begin{equation*}
		d_W(\bX^\bA,\bP^\bA)\leq\sum_{i=1}^d\left(
		\sum_{\substack{m_{1:i}\in\N_0^{i}\\m_i\neq0}}\left|q^{A_{1:i}}_{m_{1:i}}\right|
		+2\lambda(A_i)\sum_{j=1}^i\E|Z^{A_{1:i}}_j|\right).
	\end{equation*}
	Taking the supremum over all $d$-tuples of disjoint measurable sets concludes the proof.
\end{proof}
Let us now prove that the total variation distance is dominated by $d_\pi$.
Recall that the total variation distance between two point processes $\xi$ and $\zeta$ on $\X$ is
\begin{equation}
\label{eq:TVD-pointpr}
	d_{TV}(\xi,\zeta)=\sup_{B\in\cN_\X}	|\P(\xi \in B)-\P(\zeta\in B)|\,.
\end{equation}
The result is obtained by a monotone class Theorem, \cite[Theorem 1.3]{lieb2001analysis}, which is stated hereafter as a Lemma.
A monotone class $\cA$ is	a collection of sets closed under monotone 	limits, that is, for any $A_1,A_2,\ldots\in\cA$ 	with $A_n\uparrow A$ or $A_n\downarrow A$,	then $A\in\cA$.
\begin{lemma}
\label{lem:MCT}
	Let $U$ be a set and let $\cU$ be an algebra of subsets of $U$.
	Then, the monotone class generated by $\cU$ coincides with the $\sigma$-field generated by $\cU$.
\end{lemma}
\begin{proposition}
	\label{prop:TVD<OUR}
	Let $\xi $ and $\zeta$ be two point processes on $\X$ with finite intensity measure.
	Then,
	\begin{equation*}
		d_{TV}(\xi, \zeta)\leq d_{\pi}(\xi, \zeta) .
	\end{equation*}
\end{proposition}
\begin{proof}
	Let us first introduce the set of finite counting measures
	\begin{equation*}
		\sN^{<\infty}_\X=\{\nu\in\sN_\X:\nu(\X)<\infty\},
	\end{equation*}
	with the trace $\sigma$-field
	\begin{equation*}
		\cN^{<\infty}_\X=\{B\cap\sN^{<\infty}_\X: B\in\cN_\X\}.
	\end{equation*}
	As we are dealing with finite point processes, the total variation distance	is equivalently obtained if $\cN_X$ is replaced	by $\cN^{<\infty}_\X$ in \eqref{eq:TVD-pointpr}:
	\begin{equation*}
		d_{TV}(\xi,\zeta)=\sup_{B\in\cN^{<\infty}_\X}|\P(\xi \in B)-\P(\zeta\in B)|.
	\end{equation*}
	Let $\cP(\N_0^d)$ denote the power set of $\N_0^d$, that is, the	collection of all subsets of $\N_0^d$.
	For any $d\in\N$ and $M\in\cP(\N_0^d)$ 	note that $\I_M(\cdot)\in\lip^{d}(1)$, therefore
	\begin{equation}
	\label{eq:ineq2}
		d_\pi(\xi,\zeta)\geq\sup_{U\in\cU}\left|\P(\xi\in U)-\P(\zeta\in U)	\right|,
	\end{equation}
	with
	\begin{equation*}
		\cU=\left\{\left\{\nu\in\sN^{<\infty}_\X:(\nu(A_1),\dots,\nu(A_d))\in M\right\}:
		d\in\N,\, A\in\cX^d_\disj,\,M\in\cP(\N_0^d)	\right\}\subset\cN^{<\infty}_\X.
	\end{equation*}
	It can be easily verified that $\cU$ is an algebra and $\sigma(\cU)=\cN^{<\infty}_\X$.
	Moreover, by \eqref{eq:ineq2}, $\cU$ is a subset of the monotone class
	\begin{equation*}
		\left\{U\in\cN^{<\infty}_\X:\left|\P(\xi\in U)-\P(\zeta\in U)\right|	\leq d_\pi(\xi,\zeta)\right\}.
	\end{equation*}
	Lemma~\ref{lem:MCT} concludes the proof.
\end{proof}

In the last part of this section, we show that the distance $d_\pi$ is dominated by $2 d_{KR}$ when the underlying space is locally compact second countable Hausdorff (lcscH).
A topological space is second countable if its topology has a countable basis, and it is locally compact if every point has an open neighborhood whose topological closure is compact. 
Suppose that $\X$ is lcscH with Borel $\sigma$-field $\cX$.  Recall that the Kantorovich-Rubinstein distance between two point processes $\xi$ and $\zeta$ on $\X$ with finite intensity measure is given by
\begin{equation*}
	d_{KR}(\xi,\zeta)=\sup_{h\in \mathcal{L}(1)}\left| \E[h(\xi)]-\E[h(\zeta)]\right|,
\end{equation*}
where $\mathcal{L}(1)$ is the set of all measurable functions $h:\sN_\X\to \R$ that are Lipschitz continuous with respect to the total variation distance between measures
\begin{equation*}
	d_{TV,\sN_\X}(\mu,\nu)
	= \sup_{\substack{A\in\cX,\\ \mu(A),\nu(A)<\infty}}\vert \mu(A)-\nu(A)\vert,\quad \mu,\nu \in\sN_\X,
\end{equation*}
with Lipschitz constant bounded by $1$.
Since $\xi$ and $\zeta$ take values in $\sN_\X^{<\infty}$, by \cite[Theorem 1]{MR1562984} we may assume that $h$ is defined on $\sN_\X^{<\infty}$ .
\begin{lemma}
\label{lemma-KR,pi-rel}
	Let $\xi$ and $\zeta$ be two point processes with finite intensity measure on a lcscH space $\X$ with Borel $\sigma$-field $\cX$.
	Then
	\begin{align*}
		 d_{\pi}(\xi, \zeta)\leq 2 d_{KR}(\xi,\zeta).
	\end{align*}
\end{lemma}
\begin{proof}
	For $g\in\text{Lip}^d(1)$ and disjoint sets $A_1,\dots,A_d\in\cX, d\in\N,$ define $h:\sN_\X^{<\infty}\to\R$ by $h(\nu)=g(\nu(A_1),\hdots,\nu(A_d))$.
	For finite point configurations $\nu_1$ and $\nu_2$, we obtain
	\begin{align*}
		|h(\nu_1)-h(\nu_2)| &\leq | g(\nu_1(A_1),\dots,\nu_1(A_d)) - 
		g(\nu_2(A_1),\dots,\nu_2(A_d)) | 
		\\
		&\leq \sum_{i=1}^d 
		|\nu_1(A_i)-\nu_2(A_i)| \leq 2 d_{TV,\sN_\X}(\nu_1,\nu_2).
	\end{align*}
	This implies $h/2\in\mathcal{L}(1)$.
	Hence
	$|\mathbb{E}[h(\xi)]-\mathbb{E}[h(\zeta)] |\leq 2d_{KR}(\xi, \zeta)$.
	\end{proof}

\pagebreak

\section{Applications}
\label{Sec: applications}

\subsection{Sum of $m$-dependent Bernoulli random vectors}
\label{Sec:application-eq:mult-dstr}
 
In this subsection, we consider a finite family of Bernoulli random vectors $\bY^{(1)},\dots, \bY^{(n)}$ and  investigate the multivariate Poisson approximation of   $\bX=\sum_{r=1}^n \bY^{(r)}$ in the Wasserstein distance.
If the Bernoulli random vectors are i.i.d., $\bX$ has the so called multinomial distribution.
The multivariate Poisson approximation of the multinomial distribution, and more generally of the sum of  independent Bernoulli random vectors, has already been tackled by many authors in terms of the total variation distance.
Among others, we refer the reader to \cite{MR2205332,dehe:pfei1988,MR1701409,MR3647296}
and the survey \cite{MR3992498}.
Unlike the mentioned papers, we assume that  $\bY^{(1)},\dots, \bY^{(n)}$ are $m$-dependent.
Note that the case of sums of $1$-dependent random vectors has recently been treated in \cite{ekanaviius2020compound} using metrics that are weaker than the total variation distance.
To the best of our knowledge, this is the first paper where    the Poisson approximation of the sum of $m$-dependent Bernoulli random vectors is investigated in terms of the Wasserstein distance.

More precisely, for $n\in\N$, let $\bY^{(1)},\dots,\bY^{(n)}$ be Bernoulli random vectors with distributions given by 
\begin{equation}
\label{eq:distr-categRV}
\begin{aligned}
	\P(\bY^{(r)}=\mathbf{e}_j)&=p_{r,j}\in[0,1],\quad r=1,\dots,n\,,\quad j = 1,\dots,d,\\
	\P(\bY^{(r)}=\mathbf{0})&=1-\sum_{j=1}^d p_{r,j}\in[0,1],\quad r=1,\dots,n,
\end{aligned}
\end{equation}
where $\mathbf{e}_j$ denotes the vector with entry $1$ at position $j$ and entry $0$ otherwise. Assume that $\bY^{(1)},\dots,\bY^{(n)}$ are $m$-dependent for a given fixed $m\in\N_0$.
This means that  for any two subsets $S$ and $T$ of $\{1,\dots,n\}$ such that $\min(S)-\max(T)>m$, the collections $(\bY^{(s)})_{s\in S}$ and $(\bY^{(t)})_{t\in T}$ are independent.
Define the random vector $\bX=(X_1,\dots,X_d)$ as
\begin{equation}
	\label{eq:mult-dstr}
	\bX=\sum_{r=1}^n \bY^{(r)}.
\end{equation}
Note that if $\bY^{(r)},r=1,\dots,n,$ are i.i.d., then $m=0$ and $\bX$ has the multinomial distribution.
The mean vector of $\bX$ is $\E[\bX]=(\lambda_1,\dots,\lambda_d)$ with
\begin{equation}
	\label{eq:lambda}
	\lambda_j=\sum_{r=1}^n p_{r,j},\quad j=1,\dots,d.
\end{equation}
 For $k=1,\dots,n$, let $Q(k)$ be the quantity given by
\begin{equation*}
\label{eq:joint-pr}
Q(k)= \underset{\substack{r\in\{1,\dots, n\}\, : \, 1\leq \vert k-r\vert \leq m\\ i,j=1,\dots,d}}{\max}	\, \mathbb{E}\big[\I\{\bY^{(k)}=\mathbf{e}_i\}\I\{\bY^{(r)}=\mathbf{e}_j\}\big] .
\end{equation*}
We now state the main result of this subsection.
\begin{theorem}
\label{theo:multinomial}
	Let $\bX$ be as in \eqref{eq:mult-dstr}, and let $\bP=(P_1,\dots,P_d)$ be a Poisson random vector with mean $\mathbb{E}[\bP]=(\lambda_1,\dots,\lambda_d)$  given by \eqref{eq:lambda}.
	Then,
	\begin{equation*}
		d_{W}(\bX,\bP) \leq \sum_{k=1}^n \sum_{i=1}^d\bigg[ \sum_{\substack{r=1,\dots, n, \\ \vert r-k\vert \leq m}}
		p_{r,i} + 2\sum_{j=1}^{i-1}\sum_{\substack{r=1,\dots, n, \\ \vert r-k\vert \leq m}}
		p_{r,j}\bigg]p_{k,i}+ 	 2 d(d+1)m  \sum_{k=1}^nQ(k) .
	\end{equation*}
\end{theorem}
The proof of Theorem \ref{theo:multinomial} is obtained  by applying Theorem \ref{theo:1}.
When $d=1$, Equation \eqref{eq: distr of Z} corresponds to the condition required in \cite[Theorem 1.2]{pianoforte2021poisson}, which establishes sharper Poisson approximation results than the one obtained in the univariate case from Theorem \ref{theo:1}.
Therefore, for the sum of dependent Bernoulli random variables, a sharper bound for the Wasserstein distance can be derived from \cite[Theorem 1.2]{pianoforte2021poisson}, while for the total variation distance  a bound may be deduced from \cite[Theorem 1.2]{pianoforte2021poisson}, \cite[Theorem 1]{MR972770} and \cite[Theorem 1]{MR978362}.

As a consequence of Theorem~\ref{theo:multinomial},  we obtain the following result for the sum of independent Bernoulli random vectors.
\begin{corollary}
\label{cor:multinomial}
	For  $n\in\N$, let $\bY^{(1)},\dots,\bY^{(n)}$ be independent Bernoulli random vectors with distribution given by \eqref{eq:distr-categRV}, and let $\bX$ be the random vector defined by \eqref{eq:mult-dstr}.
	Let $\bP=(P_1,\dots,P_d)$ be a Poisson random vector with mean $\mathbb{E}[\bP]=(\lambda_1,\dots,\lambda_d)$  given by \eqref{eq:lambda}. Then
	\begin{equation*}
		d_{W}(\bX,\bP)\leq \sum_{k=1}^n \bigg[\sum_{i=1}^d p_{k,i}\bigg]^2 .
	\end{equation*}
\end{corollary}
In \cite[Theorem 1]{MR1701409}, a sharper bound for the total variation distance than the one obtained by Corollary~\ref{cor:multinomial} is proven.
When the vectors are identically distributed and $\sum_{j=1}^d p_{1,j}\leq \alpha/n$ for some constant $\alpha>0$, our bound for the Wasserstein distance and the one in \cite[Theorem 1]{MR1701409} for the total variation distance only differ by a constant that does not depend on $n$, $d$ and the probabilities $p_{i,j}$.
\begin{proof}[Proof of Theorem~\ref{theo:multinomial}]
	Without loss of generality we may assume that $\lambda_1,\dots,\lambda_d>0.$ Define the random vectors
	\begin{align*}
	 \mathbf{W}^{(k)}
	 &=\big(W_1^{(k)},\dots, W_d^{(k)}\big) = \sum_{\substack{r=1,\dots, n, \\ 1\leq \vert r-k\vert \leq m} } \bY^{(r)},
	 \\
	 \bX^{(k)}
	  &=\big(X_1^{(k)},\dots, X_d^{(k)} \big)= \bX-\bY^{(k)} - \mathbf{W}^{(k)} ,
	\end{align*}
	for $k=1,\dots,n$.
Let us fix $i=1,\dots,d$ and $\ell_{1:i}\in \N_0^{i}$ with $\ell_i\neq 0$.
	From straightforward calculations it follows that
	\begin{align}
		\ell_i\P(X_{1:i}=\ell_{1:i})&=\E\sum_{k=1}^n\I\{\bY^{(k)}=\mathbf{e}_i\}\I\{X_{1:i}=\ell_{1:i}\}
		\label{eq:1profMul}
		\\
		&=\E\sum_{k=1}^n\I\{\bY^{(k)}=\mathbf{e}_i\}
		\I\big\{X^{(k)}_{1:i}+ W^{(k)}_{1:i}=(\ell_{1:i-1},\ell_{i}-1)\big\} .
		\notag
	\end{align}
	Let $H^{(i)}_{\ell_{1:i}} $ and $q^{(i)}_{\ell_{1:i}}$ be the quantities given by
	\begin{align*}
		H^{(i)}_{\ell_{1:i}}
		&=\E\sum_{k=1}^n\I\{\bY^{(k)}=\mathbf{e}_i\}\I\big\{X^{(k)}_{1:i}=(\ell_{1:i-1},\ell_{i}-1)\big\} ,
		\\
		q^{(i)}_{\ell_{1:i}}&=\ell_i\P(X_{1:i}=\ell_{1:i})-H^{(i)}_{\ell_{1:i}}.
	\end{align*}
	For $i=1,\dots,d$, let $\tau_i$ be a random variable independent of $(\bY^{(r)})_{r=1}^n$	with distribution
	\begin{equation*}
		\P(\tau_i=k)=p_{k,i}/\lambda_i,\quad k=1,\dots,n\,.
	\end{equation*}
	Since $\bY^{(r)}$, $r=1,\dots,n$, are $m$-dependent, the random vectors  $\bY^{(k)}=(Y^{(k)}_1,\dots, Y^{(k)}_d)$ and $\bX^{(k)}$ are independent for all $k=1,\dots,n$. Therefore
	\begin{align*}
		H^{(i)}_{\ell_{1:i}}
		&=\sum_{k=1}^n p_{k,i}\P\big(X^{(k)}_{1:i}=(\ell_{1:i-1},\ell_{i}-1)\big)
		\\
		&=\sum_{k=1}^n p_{k,i}\P\big(X_{1:i} - W^{(k)}_{1:i} - Y^{(k)}_{1:i}  =	( \ell_{1:i-1},\ell_{i}-i)\big)
		\\
		&=\lambda_i\P\big(X_{1:i}-   W^{(\tau_i)}_{1:i} - Y^{(\tau_i)}_{1:i}=( \ell_{1:i-1},\ell_{i}-i)\big) .
	\end{align*}
	Then, by Theorem ~\ref{theo:1} we obtain  
	\begin{equation}
	\label{eq:Main-Thm-x-Mult}
		d_{W}(\bX,\bP)
		\leq\sum_{i=1}^d\bigg( \lambda_i\E\left[W_i^{(\tau_i)} + Y_i^{(\tau_i)}\right]
		+ 2\lambda_i\sum_{j=1}^{i-1}\E\left[ W_j^{(\tau_i)}+Y_j^{(\tau_i)}\right]
		+\sum_{\ell_{1:i}\in\N_{0}^d\atop \ell_i\neq 0}\big\vert q_{\ell_{1:i}}^{(i)}\big\vert\bigg).
	\end{equation}
	From \eqref{eq:1profMul} and the definition of $q_{\ell_{1:i}}^{(i)}$  it follows that  
	\begin{align*}
		| q_{\ell_{1:i}}^{(i)}| 
		&\leq\E\sum_{k=1}^n\I\{\bY^{(k)}=\mathbf{e}_i\}
		\left|\I\big\{X^{(k)}_{1:i}+ W^{(k)}_{1:i}=(\ell_{1:i-1},\ell_{i}-1)\big\}
		-\I\big\{X^{(k)}_{1:i}=(\ell_{1:i-1},\ell_{i}-1)\big\}\right|
		\\
		& \leq  \E\sum_{k=1}^n \I\{\bY^{(k)}=\mathbf{e}_i\}\I\{W_{1:i}^{(k)}\neq 0\}	\I\big\{X^{(k)}_{1:i}+ W^{(k)}_{1:i}=(\ell_{1:i-1},\ell_{i}-1)\big\}
		\\
		&\quad + \E\sum_{k=1}^n\I\{\bY^{(k)}=\mathbf{e}_i\}\I\{W_{1:i}^{(k)}\neq 0\}\I\big\{X^{(k)}_{1:i}=(\ell_{1:i-1},\ell_{i}-1)\big\}.
	\end{align*}
	Thus, by the inequality $\I\{W_{1:i}^{(k)}\neq 0\}\leq  \sum_{j=1}^i W_j^{(k)}$ we obtain 
	\begin{equation}
	\begin{split}
	 \label{eq:bnqQmult}
	 		\sum_{\ell_{1:i}\in\N_0^i \atop \ell_i\neq0} \big\vert q_{\ell_{1:i}}^{(i)}\big\vert
		& \leq 2 \E\sum_{k=1}^n \I\{\bY^{(k)}=\mathbf{e}_i\}\I\{W_{1:i}^{(k)}\neq 0\}
		\\ 
		& \leq 2 \E\sum_{k=1}^n \sum_{j=1}^i \I\{\bY^{(k)}=\mathbf{e}_i\} W_j^{(k)}\leq 4m i\sum_{k=1}^n  Q(k).
	\end{split}
	\end{equation}
	Moreover, for any $i,j=1,\dots, d$ we have  
	\begin{align*}
		\lambda_i\E \Big[W^{(\tau_i)}_j + Y^{(\tau_i)}_j \Big]
		&=\lambda_i\E \sum_{\substack{r=1,\dots, n, \\ \vert r-\tau_i\vert \leq m}}\I\{\bY^{(r)}=\mathbf{e}_j\}
		\\
		&= \sum_{k=1}^n p_{k,i}\,\E \sum_{\substack{r=1,\dots, n, \\ \vert r-k\vert \leq m}}\I\{\bY^{(r)}=\mathbf{e}_j\}
		=\sum_{\substack{k,r=1,\dots, n, \\ \vert r-k\vert \leq m}} p_{k,i}p_{r,j}.
	\end{align*}
	Together with \eqref{eq:Main-Thm-x-Mult} and \eqref{eq:bnqQmult}, this leads to
	\begin{align*}
		d_{W}(\bX,\bP)	&\leq
		\sum_{i=1}^d\sum_{\substack{k,r=1,\dots, n, \\ \vert r-k\vert \leq m}} p_{k,i}p_{r,i}
		+2\sum_{i=1}^d\sum_{j=1}^{i-1}\sum_{\substack{k,r=1,\dots, n, \\ \vert r-k\vert \leq m}} p_{k,i}p_{r,j}
		+  2d(d+1)m  \sum_{k=1}^nQ(k)
		\\
		&\leq \sum_{k=1}^n \sum_{i=1}^d\bigg[ \sum_{\substack{r=1,\dots, n, \\ \vert r-k\vert \leq m}}
		p_{r,i} + 2\sum_{j=1}^{i-1}\sum_{\substack{r=1,\dots, n, \\ \vert r-k\vert \leq m}}
		p_{r,j}\bigg]p_{k,i}+ 	 2d(d+1)m  \sum_{k=1}^nQ(k),
	\end{align*}
	which completes the proof.
\end{proof}

\subsection{Point processes with Papangelou intensity}
\label{sec:pap-int}

Let $\xi$ be a proper point process on a measurable space $(\X,\cX)$, that is, a point process that can be written as $\xi=\delta_{X_1}+\dots+\delta_{X_\tau}$, for  some random elements $X_1,X_2,\dots$ in $\X$ and a random variable $\tau\in\N_0\cup\{\infty\}$.
Note that any Poisson process can be seen as a proper point process, and that all locally finite point processes are proper if $(\X,\cX)$ is a Borel space; see e.g.\,\cite[Corollaries 3.7 and 6.5]{MR3791470}.
The so-called reduced Campbell measure $\cC$ of $\xi$ is defined on the product space $(\X\times\sN_\X,\cX\otimes\cN_\X)$ by
\begin{equation*}
	\cC(A)=\E\int_\X\I_A(x,\xi\setminus x)\,\xi(dx),\quad A\in\cX\otimes\cN_\X ,
\end{equation*}
where $\xi\setminus x$ denotes the point process $\xi-\delta_x$ if $x\in\xi$, and $\xi$ otherwise.
Let $\nu$ be a $\sigma$-finite measure on $(\X,\cX)$ and let $\P_\xi$ be the distribution of $\xi$ on $(\sN_\X,\cN_\X)$.
If $\cC$ is absolutely continuous with respect to	$\nu\otimes\P_{\xi}$,
any density $c$ of $\cC$ with respect to $\nu\otimes\P_{\xi}$ is called (a version of) the Papangelou intensity of $\xi$.
This notion was originally introduced by Papangelou in \cite{MR373000}.
In other words, $c$ is a Papangelou intensity of	$\xi$ relative to the measure $\nu$ if the Georgii–Nguyen–Zessin equation
\begin{align}
\label{GNZ}
	\E\int_\X u(x,\xi\setminus x)\,\xi(dx)=\int_\X\E[c(x,\xi)u(x,\xi)]\nu(dx) ,
\end{align}
is satisfied for all measurable functions $u:\X\times\sN_\X\rightarrow[0,\infty)$.
Intuitively $c(x,\xi)$ is a random variable that measures the interaction between $x$ and $\xi$;
as a reinforcement of this exposition,	it is well-known that if $c$ is deterministic, that is, $c(x,\xi)=f(x)$ for some  positive and measurable function $f$, then $\xi$ is a Poisson process with intensity measure $\lambda (A)=\int_{A}f(x)\nu(dx)$, $A\in \cX$; see e.g.\ \cite[Theorem 4.1]{MR3791470}.
For more details on this interpretation we refer the reader to \cite[Section 4]{decr:vass2018}.
See also \cite{Last:Otto2021} and \cite{Schu:Stu2014} for connections between Papangelou intensity and Gibbs point processes.
We show that for any proper point process $\xi$ that admits Papangelou intensity $c$ relative to a measure $\nu$, the $d_{\pi}$ distance between $\xi$ and a Poisson process with finite intensity measure $\lambda$ absolutely continuous with respect to $\nu$ can be bounded by the distance in	$L^1(\X\times\sN_\X,\nu\otimes\P_\xi)$
between $c$ and the density of $\lambda$.
For a locally compact metric space, Theorem~\ref{theo:Pap-int} yields the same bound as \cite[Theorem 3.7]{decr:vass2018}, but for the metric $d_\pi$ instead of the Kantorovich-Rubinstein distance.
\begin{theorem}
\label{theo:Pap-int}
	Let $\xi$ be a proper point process on $\X$ that admits Papangelou intensity $c$ with respect  to a $\sigma$-finite measure $\nu$ such that $\int_\X\E|c(x,\xi)|\nu(dx)<\infty$.
	Let $\eta$ be a Poisson process on $\X$ with finite intensity measure $\lambda$ having density $f$ with respect to $\nu$.
	Then
	\begin{equation*}
		d_{\pi}(\xi,\eta)\leq\int_\X\E\left|c(x,\xi)-f(x)\right|\nu(dx) .
	\end{equation*}
\end{theorem}
\begin{proof}[Proof of Theorem~\ref{theo:Pap-int}]
	The condition $\int_\X\E|c(x,\xi)|\nu(dx)<\infty$ and Equation~\eqref{GNZ} ensure that $\xi$ has finite intensity measure.
	Consider $i\in\N$ and $(A_1,\dots,A_i)\in\cX_\disj^i$.
	Hereafter,  $\xi(A_{1:i})$ is shorthand notation for $(\xi(A_1),\dots,\xi(A_i))$.
	The idea of the proof is to apply Theorem~\ref{theo:2} with the random vectors $\bZ^{A_{1:i}}$ assumed to be $\mathbf{0}$.
	In this case,
	\begin{align*}
		q^{A_{1:i}}_{m_{1:i}}
		&=m_i\P\big(\xi(A_{1:i})=m_{1:i}\big)-\lambda(A_i)
		\P\big(\xi(A_{1:i})=(m_{1:i-1},m_i-1)\big)
		\\
		&=m_i\P\big(\xi(A_{1:i})=m_{1:i}\big)
		-\int_\X\E\big[f(x)\I_{A_i}(x)
		\I\{\xi(A_{1:i})=(m_{1:i-1},m_i-1)\}\big]
		\nu(dx)
	\end{align*}
	for $m_{1:i}\in\N_0^i$ with $m_i\neq 0$, $i=1,\dots,d$.
	It follows from \eqref{GNZ} that
	\begin{align*}
		m_i\P\big(\xi(A_{1:i})=m_{1:i}\big)
		&=\E\int_\X\I_{A_i}(x)\I\{\xi\setminus x(A_{1:i})=(m_{1:i-1}, m_i -1)\}\,\xi(dx)
		\\
		&=\int_\X\E\big[c(x,\xi)\I_{A_i}(x)	\I\{\xi(A_{1:i})=(m_{1:i-1},m_i -1)\}\big]\nu(dx) ,
	\end{align*}
	hence
	\begin{equation*}
		q^{A_{1:i}}_{m_{1:i}}=	\int_\X\E\big[(c(x,\xi)-f(x))\I_{A_i}(x)
		\I\{\xi(A_{1:i})=(m_{1:i-1},m_i -1)\}\big]\nu(dx) .
	\end{equation*}
	 Theorem~\ref{theo:2} yields
	\begin{equation*}
		d_{\pi}(\xi, \eta)\leq \underset{(A_1,\dots , A_d)\in\cX^d_{\disj},d\in\N}{\sup} \,\,
		\sum_{i=1}^d \sum_{\substack{m_{1:i}\in\N_0^{i}\\m_i\neq0}}
		\left|q^{A_{1:i}}_{m_{1:i}}\right|  .
	\end{equation*}
	The inequalities
	\begin{align*}
		\sum_{m_{1:i}\in\N_0^i\atop m_i\neq 0}
		\left|q^{A_{1:i}}_{m_{1:i}}\right|
		&\leq
		\sum_{m_{1:i}\in\N_0^i, \atop m_i\neq 0}
		\int_\X\E\big[|c(x,\xi)-f(x)|\I_{A_i}(x)
		\I\{\xi(A_{1:i})=(m_{1:i-1},m_i -1)\}\big]\nu(dx)
		\\
		&\leq
		\int_\X\E\Big[|c(x,\xi)-f(x)|\I_{A_i}(x)
		\sum_{m_{1:i}\in\N_0^i \atop m_i\neq 0}
		\I\{\xi(A_{1:i})=(m_{1:i-1},m_i -1)\}\Big]\nu(dx)
		\\
		&\leq\int_\X\E\big[|c(x,\xi)-f(x)|\I_{A_i}(x)\big]\nu(dx)
	\end{align*}
	imply that
	\begin{equation*}
		\sum_{i=1}^d  \sum_{m_{1:i}\in\N_0^i\atop m_i\neq 0}
		\left|q^{A_{1:i}}_{m_{1:i}}\right|
		\leq \int_\X\E\left|c(x,\xi)-f(x)\right|\nu(dx)
	\end{equation*}
	 for any $A_{1:d}\in \cX^d_{\disj}$ with $d\in\N$.
	Thus, we obtain the assertion.
\end{proof}

\subsection{Point processes of Poisson $U$-statistic structure }
\label{Sec:U-stat}

Let $(\X,\cX)$ and $(\Y,\cY)$ be measurable spaces.
For  $k\in\N$ and a symmetric domain $D\in\cX^k$, let $g:D\to\Y$ be a symmetric measurable function, i.e., for any $(x_1,\dots,x_k)\in D$ and any index permutation $\sigma$, $(x_{\sigma(1)},\dots,x_{\sigma{(k)}})\in D$ and $g(x_1,\dots,x_k)=g(x_{\sigma(1)},\dots,x_{\sigma{(k)}})$.
Let $\eta$ be a Poisson process on $\X$ with finite intensity measure $\mu$.
We are interested in the point  process on $\Y$ given by
\begin{equation}
\label{eq:xi-U-stat}
	\xi=\frac{1}{k!}\sum_{(x_1,\dots, x_k)\in\eta^{k}_{\neq}\cap D}\delta_{g(x_1,\dots, x_k)} ,
\end{equation}
where for any $\zeta=\sum_{i\in I}\delta_{x_i}\in\sN_\X$ with $I$ at most countable, $\zeta^k_{\neq}$ denotes the collection of all $k$-tuples $(x_1,\dots,x_k)$ of points from $\zeta$ with pairwise distinct indexes.
The point process $\xi$ has a Poisson $U$-statistic structure in the sense that, for any $B\in\cY$, $\xi(B)$ is a Poisson $U$-statistic of order $k$.
We refer to the	monographs \cite{MR1472486,MR1075417}  for more details on $U$-statistics and their applications to statistics.
Hereafter we discuss the Poisson process approximation in the metric $d_\pi$ for the point process $\xi$.
We prove the exact analogue of \cite[Theorem 3.1]{decr:schu:taele}, with the Kantorovich–Rubinstein distance replaced by $d_\pi$.
Several applications of this result are presented in \cite{decr:schu:taele}, alongside with the case of underlying binomial point processes.
It is worth mentioning that \cite{decr:schu:taele} relies on a slightly less general setup: $\X$ is  assumed to be a locally compact second countable Hausdorff space, while in the present work any measurable space is allowed.

Let $\lambda$ denote the intensity measure of $\xi$, and note that, since $\mu$ is a finite measure on $\X$, by the multivariate Mecke formula $\lambda(\Y)<\infty$.
Define
\begin{equation*}
	R=\max_{1\leq i \leq k -1}\int_{\X^i} \bigg( \int_{\X^{k-i}}
	\I\{  (x_1,\dots , x_k)\in D \}\,\mu^{k-i}(d(x_{i+1},\dots , x_k)) \bigg)^2 \mu^i(d(x_1,\dots,x_i))
\end{equation*}
for $k\geq 2$, and  put $R =0$ for $k=1$.
\begin{theorem}
\label{theo:U-stat}
	Let $\xi$, $\lambda$ and $R$ be as above, and let $\gamma$ be a Poisson process on $\Y$ with intensity measure $\lambda$.
	Then,
	\begin{equation*}
		d_{\pi} ( \xi , \gamma)\leq \frac{2^{k+1}}{ k!}R.
	\end{equation*}
\end{theorem}
If the intensity measure $\lambda$ of $\xi$ is the zero measure, then the proof of Theorem~\ref{theo:U-stat} is trivial.
From now on, we assume $0<\lambda(\Y)<\infty$.
The multivariate Mecke formula yields for every $A \in\cY$ that
\begin{align*}
	\lambda(A)=\E[\xi(A)]
	=\frac{1}{k!}\E\underset{\bx\in\eta^{k}_{\neq}\cap D}{\sum}\I\{g (\bx)\in A\}
	=\frac{1}{k!}\int_{D}\I\{g (\bx)\in A\}\,\mu^k(d\bx).
\end{align*}
Define the random element $\bX^A=(X^A_1,\dots, X^A_k)$ in $\X^k$ independent of $\eta$ and distributed according to
\begin{align*}
	&\P\left(\bX^A\in B\right)	=\frac{1}{k! \lambda(A)}\int_D
	{\I\{g(\bx)\in A\}} \I\{\bx\in B \}\,\mu^k(d\bx)
\end{align*}
for all $B$ in the product $\sigma$-field of $\X^k$ when $\lambda(A)>0$,
and set $\bX^A=\bx_0$ for some $\bx_0\in\X^k$ when $\lambda(A)=0$.
For any vector $\bx=(x_1,\dots,x_k)\in\X^k$, denote by $\Delta(\bx)$ the sum of $k$ Dirac measures located by the vector components, that is
\begin{equation*}
	\Delta(\bx)=\Delta(x_1,\dots,x_k)=\sum_{i=1}^k \delta_{x_i}\,.
\end{equation*}
In what follows, for any point process $\zeta$ on $\X$, $\xi(\zeta)$ is the point process defined as in \eqref{eq:xi-U-stat} with $\eta$ replaced by $\zeta$.
Further, like in Section~\ref{sec:pap-int}, $\xi(A_{1:i})$ denotes the random vector $(\xi(A_1),\dots,\xi(A_i))$, for any $A_1,\dots,A_i\in \cY$, $i\in\N$.
\begin{proof}[Proof of Theorem~\ref{theo:U-stat}]
	For $k=1$, Theorem~\ref{theo:U-stat} is a direct consequence of \cite[Theorem 5.1]{MR3791470}. Whence, we assume $k\geq 2$. 
	Let $A_1,\dots, A_i\in\mathcal Y$ with $i\in\N$ be disjoint sets and let $m_{1:i}\in\N_0^i$ with $m_i\neq 0$.
	Suppose $\lambda(A_i)>0$.
	The multivariate Mecke formula implies that
	\begin{equation}
	\label{eq:U-stat-coupl}
	\begin{split}
		&m_i\P(\xi(A_{1:i})=m_{1:i} )  
		=\frac{1}{k!}\E\sum_{\bx\in\eta^{k}_{\neq}\cap D}
		\I\{g(\bx)\in A_i \}\I\{\xi(A_{1:i})=m_{1:i}\}
		\\
		& = \frac{1}{k!}\int_{D}\I\{g(\bx)\in A_i\}
		\P(\xi(\eta + \Delta(\bx))(A_{1:i})=m_{1:i})\,\mu^k(d\bx)
		\\
		& =\frac{1}{k!}\int_D\I\{g(\bx)\in A_i\}
		\P\left(\xi(\eta+\Delta(\bx))(A_{1:i})-\delta_{g(\bx)}(A_{1:i})
		=(m_{1:i-1},m_i -1)\right)\mu^k(d\bx)
		\\
		& =\lambda(A_i)\P\left(
		\xi\left(\eta+\Delta\left(\bX^{A_i}\right)\right)(A_{1:i})-\delta_{g\left(\bX^{A_i}\right)}
		(A_{1:i})=(m_{1:i-1},m_i -1)\right),
	\end{split}
	\end{equation}
	where the second last inequality holds true because $\delta_{g(\bx)}(A_{1:i})$ is the vector  $(0,\dots,0,1)\in\N_0^i$ when $g(\bx)\in A_i$.
	The previous identity is verified also if $\lambda(A_i)=0$.
	Hence, for
	\begin{equation*}
		\bZ^{A_{1:i}}=\xi\left(\eta+\Delta\left(\bX^{A_i}\right)\right)(A_{1:i})
		-\xi(A_{1:i})-\delta_{g\left(\bX^{A_i}\right)}(A_{1:i})\,,
	\end{equation*}
	the quantity $q_{m_{1:i}}^{A_{1:i}}$ defined by Equation \eqref{eq:hpThPP} in Theorem~\ref{theo:2} is zero.
	Note that $\bZ^{A_{1:i}}$ has nonnegative components.
	Hence, for any $d\in\N$ and $(A_1,\dots,A_d)\in\cX_\disj^{d}$,
	\begin{align*}
		\sum_{i=1}^d \lambda(A_i)\sum_{j=1}^i\E\left|\bZ_j^{A_{1:i}}\right|
		&=\sum_{i=1}^d \lambda(A_i)\sum_{j=1}^i
		\E\left[\xi\left(\eta+\Delta\left(\bX^{A_i}\right)\right)(A_j)
		-\xi(A_j)-\delta_{g\left(\bX^{A_i}\right)}(A_j)\right]
		\\
		&\leq\sum_{i=1}^d \lambda(A_i)
		\E\left[\xi\left(\eta+\Delta\left(\bX^{A_i}\right)\right)(\Y)-\xi(\Y)-1\right]
		\\
		&=\frac{1}{k!}\sum_{i=1}^d\int_D\I\{g(\bx)\in A_i\}\E\left[
		\xi(\eta+\Delta(\bx))(\Y)-\xi(\Y)-1\right] \mu^k(d\bx)
		\\
		&\leq \lambda(\Y)\E\left[\xi\left(\eta+\Delta\left(\bX^\Y\right)\right)(\Y)-\xi(\Y)-1\right].
	\end{align*}
	Thus, Theorem~\ref{theo:2} gives
	\begin{equation}
	\label{eq:conclude1}
		d_\pi(\xi,\gamma)\leq2
		\lambda(\Y)\E\left[\xi\left(\eta+\Delta\left(\bX^\Y\right)\right)(\Y)-\xi(\Y)-1\right].
	\end{equation}
	From \eqref{eq:U-stat-coupl} with $i=1$ and $A_1=\Y,$ it follows that the random variable $\xi\left(\eta+\Delta\left(\bX^\Y\right)\right)(\Y)$ is the size bias coupling of $\xi(\Y)$.
	Property \eqref{eq:sizeBias-1d}	with $f$ being the identity function and simple algebraic computations yield
	\begin{equation}
	\label{eq:conclude2}
	\begin{split}
		\E\left[\xi\left(\eta+\Delta\left(\bX^\Y\right)\right)(\Y)-\xi(\Y)-1\right]
		&=\lambda(\Y)^{-1}\left\{\E\big[ \xi(\Y)^2\big]-\lambda(\Y)^2-\lambda(\Y)\right\}
		\\
		&=\lambda(\Y)^{-1}\left\{\Var(\xi(\Y))-\lambda(\Y)\right\}.
	\end{split}
	\end{equation}
	Moreover, \cite[Lemma 3.5]{MR3161465} gives  
	\begin{equation*}	
		\Var(\xi(\Y))-\lambda(\Y)
		\leq  \sum_{i=1}^{k-1} \frac{1}{k!} \binom{k}{i} R
		\leq \frac{2^{k}-1}{k!} R\,.
	\end{equation*}
	These inequalities combined with \eqref{eq:conclude1} and \eqref{eq:conclude2} deliver the assertion.
\end{proof}

\section*{Acknowledgements}

This research was supported by Swiss National Science Foundation (grant number\break 200021\_175584).
The authors would like to thank Chinmoy Bhattacharjee, Ilya Molchanov and Matthias Schulte for valuable comments.

	\bibliographystyle{abbrv}
	\bibliography{bib4}	

\begin{thebibliography}{10}

\bibitem{MR972770}
R.~Arratia, L.~Goldstein, and L.~Gordon.
\newblock Two moments suffice for {P}oisson approximations: the {C}hen-{S}tein
  method.
\newblock {\em Ann. Probab.}, 17(1):9--25, 1989.

\bibitem{MR1092983}
R.~Arratia, L.~Goldstein, and L.~Gordon.
\newblock Poisson approximation and the {C}hen-{S}tein method.
\newblock {\em Statist. Sci.}, 5(4):403--434, 1990.

\bibitem{bar1988}
A.~D. Barbour.
\newblock Stein's method and {P}oisson process convergence.
\newblock {\em J. Appl. Probab.}, 25A:175--184, 1988.

\bibitem{MR2205332}
A.~D. Barbour.
\newblock Multivariate {P}oisson-binomial approximation using {S}tein's method.
\newblock In {\em Stein's method and applications}, volume~5 of {\em Lect.
  Notes Ser. Inst. Math. Sci. Natl. Univ. Singap.}, pages 131--142. Singapore
  Univ. Press, Singapore, 2005.

\bibitem{MR1190904}
A.~D. Barbour and T.~C. Brown.
\newblock Stein's method and point process approximation.
\newblock {\em Stochastic Process. Appl.}, 43(1):9--31, 1992.

\bibitem{MR1163825}
A.~D. Barbour, L.~Holst, and S.~Janson.
\newblock {\em Poisson approximation}, volume~2 of {\em Oxford Studies in
  Probability}.
\newblock The Clarendon Press, Oxford University Press, New York, 1992.

\bibitem{MR2274850}
A.~D. Barbour and A.~Xia.
\newblock On {S}tein's factors for {P}oisson approximation in {W}asserstein
  distance.
\newblock {\em Bernoulli}, 12(6):943--954, 2006.

\bibitem{Brow:Xia2001}
T.~C. Brown and A.~Xia.
\newblock Stein's method and birth-death processes.
\newblock {\em Ann. Probab.}, 29(3):1373--1403, 2001.

\bibitem{ChenXia2004}
L.~H.~Y. Chen and A.~Xia.
\newblock Stein's method, {P}alm theory and {P}oisson process approximation.
\newblock {\em Ann. Probab.}, 32(3B):2545--2569, 2004.

\bibitem{decr:schu:taele}
L.~Decreusefond, M.~Schulte, and C.~Th\"{a}le.
\newblock Functional {P}oisson approximation in {K}antorovich-{R}ubinstein
  distance with applications to {U}-statistics and stochastic geometry.
\newblock {\em Ann. Probab.}, 44(3):2147--2197, 2016.

\bibitem{decr:vass2018}
L.~Decreusefond and A.~Vasseur.
\newblock Stein's method and {P}apangelou intensity for {P}oisson or {C}ox
  process approximation.
\newblock arXiv:1807.02453, 2018.

\bibitem{dehe:pfei1988}
P.~Deheuvels and D.~Pfeifer.
\newblock Poisson approximations of multinomial distributions and point
  processes.
\newblock {\em J. Multivariate Anal.}, 25(1):65--89, 1988.

\bibitem{ekanaviius2020compound}
V.~Čekanavičius and P.~Vellaisamy.
\newblock Compound {P}oisson approximations in $\ell_p$-norm for sums of weakly
  dependent vectors.
\newblock {\em J. Theor. Probab.}, 2020.
\newblock DOI:10.1007/s10959-020-01042-9.

\bibitem{MR1371949}
L.~Goldstein and Y.~Rinott.
\newblock Multivariate normal approximations by {S}tein's method and size bias
  couplings.
\newblock {\em J. Appl. Probab.}, 33(1):1--17, 1996.

\bibitem{MR1472486}
V.~S. Koroljuk and Y.~V. Borovskich.
\newblock {\em Theory of {$U$}-statistics}, volume 273 of {\em Mathematics and
  its Applications}.
\newblock Kluwer Academic Publishers Group, Dordrecht, 1994.

\bibitem{Last:Otto2021}
G.~Last and M.~Otto.
\newblock Disagreement coupling of {G}ibbs processes with an application to
  {P}oisson approximation.
\newblock arXiv:2104.00737, 2021.

\bibitem{MR3791470}
G.~Last and M.~Penrose.
\newblock {\em Lectures on the {P}oisson process}, volume~7 of {\em Institute
  of Mathematical Statistics Textbooks}.
\newblock Cambridge University Press, Cambridge, 2018.

\bibitem{MR1075417}
A.~J. Lee.
\newblock {\em {$U$}-statistics}, volume 110 of {\em Statistics: Textbooks and
  Monographs}.
\newblock Marcel Dekker, Inc., New York, 1990.

\bibitem{lieb2001analysis}
E.~H. Lieb and M.~Loss.
\newblock {\em Analysis}, volume~14 of {\em Graduate Studies in Mathematics}.
\newblock American Mathematical Society, Providence, RI, second edition, 2001.

\bibitem{MR1562984}
E.~J. McShane.
\newblock Extension of range of functions.
\newblock {\em Bull. Amer. Math. Soc.}, 40(12):837--842, 1934.

\bibitem{MR3992498}
S.~Y. Novak.
\newblock Poisson approximation.
\newblock {\em Probab. Surv.}, 16:228--276, 2019.

\bibitem{MR373000}
F.~Papangelou.
\newblock The conditional intensity of general point processes and an
  application to line processes.
\newblock {\em Z. Wahrscheinlichkeitstheorie und Verw. Gebiete}, 28:207--226,
  1973/74.

\bibitem{pianoforte2021poisson}
F.~Pianoforte and M.~Schulte.
\newblock Poisson approximation with applications to stochastic geometry.
\newblock arXiv:2104.02528, 2021.

\bibitem{MR3161465}
M.~Reitzner and M.~Schulte.
\newblock Central limit theorems for {$U$}-statistics of {P}oisson point
  processes.
\newblock {\em Ann. Probab.}, 41(6):3879--3909, 2013.

\bibitem{MR1701409}
B.~Roos.
\newblock On the rate of multivariate {P}oisson convergence.
\newblock {\em J. Multivariate Anal.}, 69(1):120--134, 1999.

\bibitem{MR1978097}
B.~Roos.
\newblock Poisson approximation of multivariate {P}oisson mixtures.
\newblock {\em J. Appl. Probab.}, 40(2):376--390, 2003.

\bibitem{MR3647296}
B.~Roos.
\newblock Refined total variation bounds in the multivariate and compound
  {P}oisson approximation.
\newblock {\em ALEA Lat. Am. J. Probab. Math. Stat.}, 14(1):337--360, 2017.

\bibitem{Schu2009}
D.~Schuhmacher.
\newblock Stein's method and {P}oisson process approximation for a class of
  {W}asserstein metrics.
\newblock {\em Bernoulli}, 15(2):550--568, 2009.

\bibitem{Schu:Stu2014}
D.~Schuhmacher and K.~Stucki.
\newblock Gibbs point process approximation: total variation bounds using
  {S}tein's method.
\newblock {\em Ann. Probab.}, 42(5):1911--1951, 2014.

\bibitem{MR978362}
R.~L. Smith.
\newblock Extreme value theory for dependent sequences via the {S}tein-{C}hen
  method of {P}oisson approximation.
\newblock {\em Stochastic Process. Appl.}, 30(2):317--327, 1988.

\bibitem{Vill2009}
C.~Villani.
\newblock {\em Optimal transport}, volume 338 of {\em Grundlehren der
  Mathematischen Wissenschaften}.
\newblock Springer-Verlag, Berlin, 2009.

\end{thebibliography}
\end{document}